\documentclass[10pt,letterpaper]{article}
\usepackage{amsfonts}
\usepackage[left=1in,top=1in,right=1in,bottom=1in]{geometry}
\usepackage{setspace}
\usepackage{bbm}
\usepackage{amsmath,amsthm,amssymb}
\usepackage{mathrsfs}
\usepackage{graphicx}
\usepackage{color}
\usepackage{soul}
\usepackage{enumitem}
\setlist{nolistsep}

\newcommand{\N}{\mathbb{N}}

\newcommand{\R}{\mathbb{R}}
\newcommand{\C}{\mathbb{C}}
\newcommand{\Mat}[2]{\textnormal{M}_{#1}\left(#2\right)}
\newcommand{\Diag}[2]{\textnormal{D}_{#1}\left(#2\right)}
\newcommand{\dif}{\textrm{d}}

\newcommand{\E}[1]{{\mathbb E}\left(#1\right)}

\newcommand{\Eop}[1]{{E}\left(#1\right)}
\newcommand{\EOP}{{E}}
\newcommand{\F}[1]{\varphi\left(#1\right)}
\newcommand{\FI}{\varphi}
\newcommand{\calA}{{\mathcal A}}
\newcommand{\calB}{{\mathcal B}}
\newcommand{\calC}{{\mathcal C}}

\newcommand{\calF}{{\mathcal F}}
\newcommand{\calG}{{\mathcal G}}

\newcommand{\TR}{\textrm{tr}}
\newcommand{\trd}[1]{\textrm{tr}_d\left(#1\right)}

\newcommand{\NC}{{\textnormal NC}}

\newtheoremstyle{DefinitionStyle}  
  {3pt}        
  {9pt}        
  {}           
  {0pt}        
  {\bfseries}  
  {.}          
  {3pt}        
  {}           
\theoremstyle{DefinitionStyle}
\newtheorem{definition}{Definition}

\newtheorem{lemma}{Lemma}
\newtheorem{proposition}{Proposition}
\newtheorem{theorem}{Theorem}
\newtheorem{corollary}{Corollary}

\newtheorem{example}{Example}

\author{Mario Diaz\thanks{Department of Mathematics and Statistics, Queen's University, Kingston, ON, Canada, \textit{13madt@queensu.ca}}}
\title{On Random Operator-Valued Matrices: Operator-Valued Semicircular Mixtures and Central Limit Theorem}
\date{\today}

\begin{document}

\maketitle


\begin{abstract}
Motivated by a random matrix theory model from wireless communications, we define random operator-valued matrices as the elements of $L^{\infty-}(\Omega,{\mathcal F},{\mathbb P}) \otimes {\textnormal M}_d({\mathcal A})$ where $(\Omega,{\mathcal F},{\mathbb P})$ is a classical probability space and $({\mathcal A},\varphi)$ is a non-commutative probability space. A central limit theorem for the mean $\textnormal{M}_d(\mathbb{C})$-valued moments of these random operator-valued matrices is derived. Also a numerical algorithm to compute the mean ${\textnormal M}_d({\mathbb C})$-valued Cauchy transform of operator-valued semicircular mixtures is analyzed.

\end{abstract}

\section{Introduction}
\label{Section:Introduction}

In this section we motivate the study of what we call random operator-valued matrices. In particular, we generalize a random matrix model used in wireless communications. This will also provide us a natural link between random matrix ensembles and random operator-valued matrices. In the following section we summarize the notation and setting of this paper.

Recent developments in operator-valued free probability theory \cite{Speicher2007,Belinschi2013, Belinschi2014} have made possible to analyze a variety of wireless communication systems \cite{Far2008,Speicher2012,Diaz2014}. In the developing area of massive multiantenna systems \cite{Larsson2014} the dimension of the random matrix modelling the system is in the order hundreds or thousands. These large sizes suggest that the behavior of the spectrum of these matrices is very close to their asymptotic models, e.g. free deterministic equivalents \cite{Speicher2012} and operator-valued equivalents \cite{Diaz2014}.

Roughly speaking, random operator-valued matrices (models) is a special class of random variables with values over the matrices with coefficients in some non-commutative algebra. This contrasts with the classical models studied before, e.g. \cite{Speicher2012,Diaz2014}, where the asymptotic models are non-random operator-valued matrices over some non-commutative algebra. From an applied point of view, this extra randomness may reflect the statistical variations of the channel in a scale of time bigger than a period of use. Therefore this kind of model may be relevant to study properties of channels that depend on statistics that change over large periods of time.

Random operator-valued matrices are also a natural object from a mathematical point of view. It is known that if operator-valued matrices, i.e. elements in $\Mat{d}{\calA}$ where $(\calA,\FI)$ is a non-commutative probability space, have entries free over $\calA$ then they are free over $\Mat{d}{\C}$. Thus, if we have two fixed families $\{a_{i,j}\}_{i,j=1}^d$ and $\{b_{i,j}\}_{i,j=1}^d$ free over $\calA$ and we construct operator-valued matrices ${\bf A}=f(a_{i,j};i,j)$ and ${\bf B}=g(b_{i,j};i,j)$ for some suitable functions $f,g:\calA^{d^2}\to\Mat{d}{\calA}$, then ${\bf A}$ and ${\bf B}$ will be free over $\Mat{d}{\C}$. If instead of a pair of functions we have a pair of families of functions $\{f_\omega\}_{\omega\in\Omega}$ and $\{g_\omega\}_{\omega\in\Omega}$ for some probability space $(\Omega,\calF,\mathbb{P})$, then ${\bf A}={\bf A}(\omega)=f_\omega(a_{i,j};i,j)$ and ${\bf B}={\bf B}(\omega)=g_\omega(b_{i,j};i,j)$ are random variables with values over operator-valued matrices, thus the name random operator-valued matrices. Of course some measurability conditions should be satisfied, but in the context of this paper this requirement will be clearly satisfied. If the families $\{f_\omega\}_{\omega\in\Omega}$ and $\{g_\omega\}_{\omega\in\Omega}$ are suitable then the freeness is preserved almost everywhere, which makes possible to work in the realm of operator-value free probability theory almost everywhere. This is similar to classical probability where we work in the realm of real analysis almost everywhere. In this sense, here we are dealing with a probabilistic version of operator-value free probability theory. We will not work at this too general level of abstraction, in what follows we construct a particular class of objects for which the pointwise operator-valued behavior is more structured.

Returning to the wireless communication context, the random matrix of interest $H_N$ have the following form
\begin{align*}
H_N &= \left(
\begin{matrix}
A_{1,1} X_N^{(1,1)} & \cdots & A_{1,d} X_N^{(1,d)}\\
\vdots & \ddots & \vdots\\
A_{d,1} X_N^{(d,1)} & \cdots & A_{d,d} X_N^{(d,d)}
\end{matrix}
\right)
\end{align*}
where $A=(A_{i,j})_{i,j=1}^d$ is a selfadjoint $d\times d$ random matrix and $\left\{X_N^{(i,j)} \mid 1\leq i,j\leq d\right\}$ is a family of independent $N\times N$ standard complex Gaussian matrices such that $X_N^{(j,i)}=\left(X_N^{(i,j)}\right)^*$ for all $1\leq i,j\leq d$. The asymptotic analysis will be done with respect to $N$ while $d$ remain fixed (see \cite{Diaz2014} for further details on this kind of models). The matrix $A$ is assumed to be independent of $X_N^{(i,j)}$ for all $1\leq i,j\leq d$ and such that all its entries belong to $L^{\infty-}(\Omega,\calF,\mathbb{P})$. In fact, the extra randomness added to this model with respect to the one analyzed in \cite{Diaz2014} comes from this matrix $A$.

Let $F^{H_N}$ be empirical eigenvalue distribution of $H_N$. Some of the quantities of interest in the wireless communication context are given by
\begin{equation*}
I_N:=\E{\int_{-\infty}^\infty f(x) \dif F^{H_N}(x)}
\end{equation*}
for some non-negative, continuous and bounded function $f$. By a standard argument we have then
\begin{equation*}
I_N=\int_{-\infty}^\infty f(x) \dif F_N(x)
\end{equation*}
where $F_N$ is the mean eigenvalue distribution of $H_N$, i.e. $F_N(x) = \E{F^{H_N}(x)}$ for all $x\in\R$.

Let $(\Omega,\calF,\mathbb{P})$ be the underlying probability space. Suppose that the previous hypothesis on $H_N$ are satisfied\footnote{One way to construct such a family is: a) construct a probability space where the random matrix $A$ exists, b) construct another probability space where the family of $\{X_N^{(i,j)}\mid N\in\N,1\leq i,j\leq N\}$ exists and c) in the product of these spaces take the inclusions associated to aforementioned random variables.} for all $N\in\N$, then for almost every $\omega\in\Omega$ we have that \cite{Hiai2000}
\begin{equation}
\label{Eq:ConvergenceDistributionH}
H_N(\omega) \stackrel{\textnormal{dist}}{\longrightarrow}
\left(
\begin{matrix}
A_{1,1}(\omega) {\bf X}_{1,1} & \cdots & A_{1,d}(\omega) {\bf X}_{1,d}\\
\vdots & \ddots & \vdots\\
A_{d,1}(\omega) {\bf X}_{d,1} & \cdots & A_{d,d}(\omega) {\bf X}_{d,d}
\end{matrix}
\right) =: {\bf H}(\omega)
\end{equation}
where $\{{\bf X}_{i,j} \mid 1\leq i,j\leq d\}$ is a free circular family in a non-commutative probability space $(\calA,\varphi)$ with ${\bf X}_{i,j}^*={\bf X}_{j,i}$ for all $1\leq i,j\leq d$ and thus ${\bf X}=({\bf X}_{i,j})_{i,j=1}^d\in\Mat{d}{\calA}$ is an operator-valued matrix. Observe that ${\bf H}(\omega)= A(\omega) \circ {\bf X}$ where $\circ$ denotes the Hadamard or entrywise product. By convergence in distribution we mean that
\begin{equation*}
\lim_{N\to\infty} ({\rm I}\otimes\TR_{N})(H_N(\omega)^m) \stackrel{\textnormal{a.s.}}{=} \Eop{{\bf H}(\omega)^m}
\end{equation*}
for all $m\in\N$ where $\TR_{N}$ is the normalized trace in $\Mat{N}{\C}$ and $E:={\rm I}\otimes\FI$. Up to this point we described the behavior of ${\bf H}(\omega)$ for a fixed $\omega\in\Omega$. We can propose several abstract spaces in which the expression ${\bf H}:=A\circ{\bf X}$ have sense, in the next paragraph we construct such a space and describe the relations that the expectation w.r.t. $\mathbb{P}$, $\mathbb{E}$, and $\EOP$ should have.

To find the intended relations first consider a $N\times N$ block of the matrix $H_N^m$ for $m\in\N$, say the $i_1,i_m$-block for some $i_1,i_m\in\{1,\ldots,d\}$. Such a block is indeed a random matrix, so the natural linear functional to study is $\mathbb{E}\circ\TR_N$. The block under study is the sum of random matrices of the form
\begin{equation*}
A_{i_1,i_2}\cdots A_{i_{m-1},i_m}X_N^{(i_1,i_2)}\cdots X_N^{(i_{m-1},i_m)}
\end{equation*}
with $i_2,\ldots,i_{m-1}\in\{1,\ldots,d\}$. The independence between $A$ and the matrices $\{X_N^{(i,j)} \mid 1\leq i,j\leq d\}$ implies that
\begin{align*}
\lim_{N\to\infty} \E{\TR_N\left(A_{i_1,i_2}\cdots A_{i_{m-1},i_m}X_N^{(i_1,i_2)}\cdots X_N^{(i_{m-1},i_m)}\right)} &= \E{A_{i_1,i_2}\cdots A_{i_{m-1},i_m}} \F{{\bf X}_{i_1,i_2}\cdots {\bf X}_{i_{m-1},i_m}}\\
&= (\mathbb{E}\otimes\FI)(A_{i_1,i_2}\cdots A_{i_{m-1},i_m}\otimes {\bf X}_{i_1,i_2}\cdots {\bf X}_{i_{m-1},i_m}).
\end{align*}

This suggests that the entries of ${\bf H}$ may belong to $(L^{\infty-}(\Omega,\calF,\mathbb{P}),{\mathbb E})\otimes(\calA,\FI)$. Therefore, ${\bf H}$ can be thought as an element in the space $(L^{\infty-}(\Omega,{\mathcal F},{\mathbb P}),\mathbb{E}) \otimes (\Mat{d}{\calA},\EOP)$. From the algebraic construction of the previous tensor product, $\mathbb{E}\otimes 1_\calA$ and $1\otimes\FI$ commute when applied to elements in $L^{\infty-}(\Omega,\calF,\mathbb{P})\otimes\calA$, $(1\otimes\EOP)(A)=A$ for all random matrix $A$ and $(\mathbb{E}\otimes{\rm I})({\bf X})={\bf X}$ for all ${\bf X}$ (non-random) operator-valued matrix and thus $1\otimes\EOP$ and $\mathbb{E}\otimes{\rm I}$ commute. If $\calB$ is a sub $\sigma$-algebra of $\calF$, from the definition of conditional expectation we have that $(\mathbb{E}\otimes 1_\calA)(\E{A\mid \calB}\otimes x)=(\mathbb{E}\otimes 1_\calA)(A\otimes x)$ for $A\in L^{\infty-}(\Omega,\calF,\mathbb{P})$ and $x\in\calA$, and in particular taking conditional expectation with respect to $\calB$ (in the first coordinate) and applying $1\otimes\FI$ commute in $L^{\infty-}(\Omega,\calF,\mathbb{P})\otimes\calA$.

By the independence between $A$ and the matrices $\{X_N^{(i,j)} \mid 1\leq i,j\leq d\}$, the convergence in equation (\ref{Eq:ConvergenceDistributionH}) also holds in mean and thus
\begin{equation*}
\lim_{N\to\infty} \E{\TR_{dN}(H_N^m)} = \E{\trd{\Eop{{\bf H}^m}}}
\end{equation*}
for all $m\in\N$, or equivalently
\begin{equation*}
\lim_{N\to\infty} \int_\R t^m \dif F_N(t) = \int_\R t^m \dif F(t)
\end{equation*}
for all $m\in\N$, where $F$ is the mean analytical distribution of ${\bf H}$ (see Definition \ref{Def:MeanAnalyticaDistribution} and Proposition \ref{Prop:Fubini}). If $F$ is determined by its moments, the convergence of the moments implies that $F_N \Rightarrow F$ and in particular $I_N\to I:=\int_\R f(t) \dif F(t)$. Therefore we can take $I$ as an approximation for $I_N$ and focus on the mean analytical distribution $F$ of ${\bf H}$.

The variance of the entries of $X$ can be absorbed by $A$, so without loss of generality we assume that $\F{{\bf X}_{i,j}{\bf X}_{k,l}}=\delta_{i,l}\delta_{j,k}$. Likewise, if $A_{i,k}=|A_{i,k}|\exp({\rm i}\arg(A_{i,k}))$ then the equality in distribution as complex random matrices $\exp({\rm i}\arg(A_{i,k}))X_N^{(i,k)}\stackrel{d}{=}X_N^{(i,k)}$ and the independence between $A$ and the family $\{X_N^{(i,j)}\mid 1\leq i,j\leq d\}$ show that the argument of $A_{i,k}$ can be absorbed by ${\bf X}_{i,k}$, so without loss of generality we assume that the entries of $A$ are positive random variables.

It is important to remark that the random operator-valued matrix of the form $A\circ{\bf X}$ as described above can be thought as the mixture by $A$ of the operator-valued semicircular element (over $\Mat{d}{\C}$) ${\bf X}$. Due to the important role of the classical Gaussian mixture in several applications, this analogy provides a general motivation for the study of elements of the aforementioned form.

\section{Notation and Setting}

Throughout this paper $d$ will be a fixed positive integer. $\Mat{d}{\C}$ denotes the $d\times d$ matrices over the complex numbers and $\Diag{d}{\C}$ the subset of diagonal complex matrices. The notation $A_{i,k}$ or $(A)_{i,k}$ denotes the $i,k$-entry of the matrix $A$. We set $\Mat{d}{\C}^\pm=\{B\in\Mat{d}{\C} \mid 0<\pm\Im(B)\}$ where $\Im(B):=\dfrac{B-B^*}{2i}$, and in particular $\C^\pm=\Mat{1}{\calC}^\pm$. Also, $\circ$ denotes the pointwise or Hadamard product of matrices.

In what follows $(\Omega,\calF,\mathbb{P})$ will be a probability space with expectation $\mathbb{E}$. We will denote by $(\calA,\FI)$ a non-commutative probability space where $1_\calA$ is the unit in $\calA$. The algebra $(\Mat{d}{\calA},\EOP)$ where $E:={\rm I}\otimes\FI:\Mat{d}{\calA}\equiv\Mat{d}{\C}\otimes\calA\to\Mat{d}{\C}$ is called the operator-valued matrices. The algebra of random operator-valued matrices is then defined to be $(L^{\infty-}(\Omega,{\mathcal F},{\mathbb P}),\mathbb{E}) \otimes (\Mat{d}{\calA},\EOP)$. Abusing of the notation, we also use $\EOP$ to denote the map $1\otimes({\rm I}\otimes\FI):L^{\infty-}(\Omega,{\mathcal F},{\mathbb P}) \otimes \Mat{d}{\calA}\to\Mat{d}{\C}$. In the same spirit, we use ${\mathbb E}$ to denote both expectation in $(\Omega,\calF,{\mathbb P})$ and the map ${\mathbb E}\otimes {\rm I}:L^{\infty-}(\Omega,{\mathcal F},{\mathbb P}) \otimes \Mat{d}{\calA}\to\Mat{d}{\calA}$.  We will use lower case letters to denote scalars, upper case letters to denote both matrices and random matrices, and bold upper case letters to denote both operator-valued matrices and random operator-valued matrices.

\section{Definitions}
\label{Section:MainResults}

Recall the following definitions from \cite{Mingo2014}.

\begin{definition}
Let ${\bf X}^{(1)}$ and ${\bf X}^{(2)}$ be operator-valued elements in $\Mat{d}{\calA}$. We say that ${\bf X}^{(1)}$ and ${\bf X}^{(2)}$ are free over $\Mat{d}{\C}$ if
\begin{equation*}
\Eop{p_1({\bf X}^{(i_1)})p_2({\bf X}^{(i_2)})\cdots p_k({\bf X}^{(i_k)})}=0
\end{equation*}
whenever $k\in\N$, $i_j\neq i_{j+1}$ for all $1\leq j <k$ and $\{p_i \mid 1\leq i\leq k\}$ are non commutative polynomials over $\Mat{d}{\C}$ such that $\Eop{p_j\left({\bf X}^{(i_j)}\right)}={\bf 0}$ for all $1\leq j\leq k$.
\end{definition}

\begin{definition}
\label{Def:SemicircularElement}
We say that ${\bf X}\in\Mat{d}{\calA}$ is an operator-valued semicircular element over $\Mat{d}{\C}$ if
\begin{equation*}
\Eop{{\bf X}^m} = \sum_{\pi\in\NC_2(m)} \kappa_\pi({\bf X})
\end{equation*}
where $\NC_2(m)$ are the non-crossing pairings of the set $\{1,\ldots,m\}$ and $\kappa_\pi({\bf X})$ is defined recursively using the nested structure of $\pi$ and the function $\eta:\Mat{d}{\C}\to\Mat{d}{\C}$ given by $\eta(B)=\Eop{{\bf X}B{\bf X}}$ for $B\in\Mat{d}{\C}$.
\end{definition}

\begin{example}
In the notation of the previous definition, if $\pi=\{(1,6),(2,3),(4,5),(7,8)\}$ then
\begin{align*}
\kappa_\pi({\bf X}) &= \eta(\eta({\rm I})\eta({\rm I})) \eta({\rm I}).
\end{align*}
\end{example}

\begin{definition}
\label{Def:SemicircularMixture}
We define an operator-valued semicircular mixture, semicircular mixture for short, as a random operator-valued matrix ${\bf H}$ such that ${\bf H}=A\circ{\bf X}$ where $A$ is a $d\times d$ selfadjoint random matrix with non-negative entries in $L^{\infty-}(\Omega,\calF,\mathbb{P})$ and ${\bf X}$ is a selfadjoint operator-valued matrix such that $\{{\bf X}_{i,j} \mid 1\leq i,j\leq d\}$ is a free circular family up to symmetry conditions\footnote{In particular, the condition ${\bf X}_{i,i}={\bf X}_{i,i}^*$ implies that the diagonal elements are semicircular non-commutative random variables.}.
\end{definition}

Recall that any random operator-valued matrix ${\bf Z}$, i.e. an element in $L^{\infty-}(\Omega,\calF,\mathbb{P})\otimes\Mat{d}{\calA}$, can be written as $\sum_{i=1}^n a_i \otimes {\bf Z}_i$ for some $n\in\N$ and $a_1,\ldots,a_n\in L^{\infty-}(\Omega,\calF,\mathbb{P})$ and ${\bf Z}_1,\ldots,{\bf Z}_n\in\Mat{d}{\calA}$.

\begin{definition}
\label{Def:IndependenceFreeness}
Let ${\bf Z}^{(1)}$ and ${\bf Z}^{(2)}$ be two elements in $L^{\infty-}(\Omega,\calF,\mathbb{P})\otimes\Mat{d}{\calA}$. We say that ${\bf Z}^{(1)}$ and ${\bf Z}^{(2)}$ are independent w.r.t. $\mathbb{P}$ if there exist $n_1,n_2\in\N$, $a^{(1)}_1,\ldots,a^{(1)}_{n_1},a^{(2)}_1,\ldots,a^{(2)}_{n_2}\in L^{\infty-}(\Omega,\calF,\mathbb{P})$ and ${\bf Z}^{(1)}_1,\ldots,{\bf Z}^{(1)}_{n_1},{\bf Z}^{(2)}_1,\ldots,{\bf Z}^{(2)}_{n_2}\in\Mat{d}{\calA}$ such that ${\bf Z}^{(1)}=\sum_{i=1}^{n_1} a^{(1)}_i\otimes{\bf Z}^{(1)}_i$, ${\bf Z}^{(2)}=\sum_{i=1}^{n_2} a^{(2)}_i\otimes{\bf Z}^{(2)}_i$ and $\{a^{(1)}_1,\ldots,a^{(1)}_{n_1}\}$ and $\{a^{(2)}_1,\ldots,a^{(2)}_{n_2}\}$ are independent families of random variables. Similarly, we say that ${\bf Z}^{(1)}$ and ${\bf Z}^{(2)}$ are free over $\Mat{d}{\C}$ if the families $\{{\bf Z}^{(1)}_1,\ldots,{\bf Z}^{(1)}_{n_1}\}$ and $\{{\bf Z}^{(2)}_1,\ldots,{\bf Z}^{(2)}_{n_2}\}$ are free over $\Mat{d}{\C}$.
\end{definition}

The linearity of $\EOP$ implies that if ${\bf Z}^{(1)}$ and ${\bf Z}^{(2)}$ are free over $\Mat{d}{\C}$ (in the sense of the previous definition), for any non-commutative polynomial in two variables $p$ the expression $\Eop{p({\bf Z}^{(1)},{\bf Z}^{(2)})}$ equals to the same expression as if ${\bf Z}^{(1)}$ and ${\bf Z}^{(2)}$ were two (non-random) operator-valued elements free over $\Mat{d}{\C}$, e.g. $\Eop{{\bf Z}^{(1)}{\bf Z}^{(2)}{\bf Z}^{(1)}}=\Eop{{\bf Z}^{(1)}\Eop{{\bf Z}^{(2)}}{\bf Z}^{(1)}}$.

\begin{definition}
\label{Def:Centered}
Let ${\bf Z}$ be a random operator-valued matrix. We say that ${\bf Z}$ is centered if there exists $n\in\N$, $a_1,\ldots,a_n\in L^{\infty-}(\Omega,\calF,\mathbb{P})$ and ${\bf Z}_1,\ldots,{\bf Z}_n\in\Mat{d}{\calA}$ such that ${\bf Z}=\sum_{i=1}^n a_i\otimes{\bf Z}_i$ and $\Eop{{\bf Z}_i}=0$ for all $1\leq i\leq n$.
\end{definition}

Observe that centered is used to refer just a property of the operator-valued part. The reason for doing this is that indeed we will not need any centeredness assumption for the random part. This reflects the fact that the operator-valued part {\it dominates} the random part in terms of the impact to the overall behavior of a random operator-valued matrix.

Recall the following standard definitions.

\begin{definition}
The scalar-valued Cauchy transform $g_F:\C^+\to\C^-$ of a probability distribution $F$ is defined for $z\in\C^+$ by
\begin{equation*}
g_F(z) = \int_\R \frac{1}{z-t} \dif F(t).
\end{equation*}
\end{definition}

\begin{definition}
Let ${\bf X}\in\Mat{d}{\calA}$ be an operator-valued matrix, we define its $\Mat{d}{\C}$-valued Cauchy transform $G_{\bf X}(B)$ for $B\in\Mat{d}{\C}$ such that $B-{\bf X}$ is invertible by
\begin{equation*}
G_{\bf X}(B) = \Eop{(B-{\bf X})^{-1}}.
\end{equation*}
\end{definition}

Following the pointwise or almost sure philosophy discussed early, the next definition is the straightforward generalization of the previous ones to our random operator-valued context.

\begin{definition}
Let ${\bf H}\in L^{\infty-}(\Omega,{\mathcal F},{\mathbb P}) \otimes \Mat{d}{\calA}$ be a random operator-valued matrix, we define its mean $\Mat{d}{\C}$-valued Cauchy transform $\calG_{\bf H}$ by
\begin{equation*}
\calG_{\bf H}(B) = \E{G_{\bf H}(B)}
\end{equation*}
for $B\in\Mat{d}{\C}$ such that $B-{\bf H}$ is invertible almost surely.
\end{definition}

Observe that $G_{\bf H}(B)$ is random matrix as it depends on ${\bf H}$. It is also possible to write $\calG_{\bf H}(B)$ as $(\mathbb{E}\otimes\EOP)((B-{\bf H})^{-1})$. However, we prefer the form used in the previous definition since, as we will see throughout the paper, the behavior of $\EOP$ {\it dominates} the one of $\mathbb{E}$.

Given a selfadjoint ${\bf X}\in\Mat{d}{\calA}$, the function $h:\C^+\to\C$ defined by $h(z)=\trd{G_{\bf X}(z{\rm I})}$ satisfies that $h(\C^+)\subset\C^-$ and $\lim_{y\to\infty} iyh(iy)=1$. Therefore \cite{Mingo2014}, there exists a unique probability distribution $F^{\bf X}$ on $\R$ such that
\begin{equation*}
h(z) = \int_\R (z-t)^{-1} \dif F^{\bf X}(t).
\end{equation*}
This observation gives sense to the following definition and its natural generalization to the random operator-valued context.

\begin{definition}
\label{Def:AnalyticaDistribution}
Let ${\bf X}\in\Mat{d}{\calA}$ be selfadjoint. We define the analytical distribution of ${\bf X}$ as the unique probability distribution $F^{\bf X}$ on $\R$ such that
\begin{equation*}
\trd{G_{\bf X}(z{\rm I})} = \int_\R (z-t)^{-1} \dif F^{\bf X}(t)
\end{equation*}
for $z\in\C^+$.
\end{definition}

\begin{definition}
\label{Def:MeanAnalyticaDistribution}
Let ${\bf H}\in L^{\infty-}(\Omega,{\mathcal F},{\mathbb P}) \otimes \Mat{d}{\calA}$ be a selfadjoint random operator-valued matrix. We define the mean analytical distribution of ${\bf H}$ as the probability distribution $F$ on $\R$ such that $F(x):=\E{F^{\bf H}(x)}$ where $F^{{\bf H}(\omega)}$ is the analytical distribution of ${\bf H}(\omega)$.
\end{definition}

\section{Mean $\Mat{d}{\C}$-Valued Cauchy Transform of Semicircular \-Mixtures}

Let ${\bf H}=A\circ{\bf X}$ be a semicircular mixture and $F$ its mean analytical distribution. Observe that $F^{\bf H}$ is a random probability distribution on $\R$ as it depends on $A$. The definition of $F$ (see Definition \ref{Def:MeanAnalyticaDistribution}) requires averaging (w.r.t. $\mathbb{P}$) the analytical distributions $F^{\bf H}$. The next proposition shows that averaging these distributions and then taking the Cauchy transform is the same as averaging the corresponding Cauchy transforms. Additionally, this proposition proves that this is also true for the moments, i.e. the moments of the mean analytical distribution of ${\bf H}$ are equal to the mean scalar-valued moments of ${\bf H}$, and they exist. In particular, this shows that the mean analytical distribution is the right object to study the behavior of ${\bf H}$.

\begin{proposition}
\label{Prop:Fubini}
Let ${\bf H}=A\circ{\bf X}$ be a semicircular mixture and $F$ its mean analytical distribution. Then the scalar-valued Cauchy transform $g_F:\C^+\to\C$ of $F$ is given by
\begin{align*}
g_F(z) = \trd{\calG_{\bf H}(z{\rm I})}.
\end{align*}
For all $m\in\N$ the $m$-th moment of $F$ exists and
\begin{equation*}
\int_\R t^m \dif F(t) = \E{\trd{\Eop{{\bf H}^m}}}.
\end{equation*}
\end{proposition}

The previous proposition imply that $g_F$ can be obtained from the mean $\Mat{d}{\C}$-valued Cauchy transform $\calG_{\bf H}$ of ${\bf H}$. By the Stieltjes inversion theorem, it is enough to compute $\calG_{\bf H}$ to obtain $F$. Therefore we will focus on $\calG_{\bf H}$ in what follows.

\begin{definition}
\label{Def:Eta}
Let ${\bf H}=A\circ{\bf X}$ be a semicircular mixture. We define the random mapping $\eta:\Diag{d}{\C}\to\Diag{d}{\C}$ for $D\in\Diag{d}{\C}$ by
\begin{equation*}
\eta(D)_{i,i} = \sum_{j=1}^d A_{i,j}^2 D_{j,j}
\end{equation*}
for all $1\leq i\leq d$.
\end{definition}

\begin{theorem}
\label{Thm:Elim}
Let ${\bf H}=A\circ{\bf X}$ be a semicircular mixture. Then
\begin{equation*}
\calG_{\bf H}(z{\rm I}) = \lim_{n\to\infty} \E{T_{z{\rm I}}^{\circ n}(W)}
\end{equation*}
for all $z\in\C^+$ and all $W\in\Mat{d}{\C}^-$ where $T_{z{\rm I}}(D)=\left(z{\rm I}-\eta(D)\right)^{-1}$ for $D\in\Diag{d}{\C}$.
\end{theorem}

The previous theorem is a straightforward extension of the operator-valued version in \cite{Speicher2007}. In particular, this extension shows that limit and expectation commute. This commutativity implies that taking expectation or limit first does not matter when computing $\calG_{\bf H}$. The latter suggests that any reasonable numerical method will approximate $\calG_{\bf H}$ robustly. In particular, the following routine is an example of such a method.

By definition, for $D\in\Diag{d}{\C}$ we have that $T_{z{\rm I}}^{\circ n+1}(D)=T_{z{\rm I}}(T_{z{\rm I}}^{\circ n}(D))$, that is
\begin{equation}
\label{Eq:Recurrence}
T_{z{\rm I}}^{\circ n+1}(D)_{i,i} = \left(z-\sum_{j=1}^d A_{i,j}^2 T_{z{\rm I}}^{\circ n}(-{\rm i}{\rm I})_{j,j}\right)^{-1}
\end{equation}
for $i\in\{1,\ldots,d\}$. So we can approximate $\calG_{\bf H}(z{\rm I})$ using Monte-Carlo method as follows
\begin{itemize}
\item[-] Put $S={\bf 0}\in\Diag{d}{\C}$ and iterate $M$ times the following subroutine;
  \begin{itemize}
  \item[--] Pick random values for $A$;
  \item[--] Define $T_{z{\rm I}}^{\circ 0}=-{\rm i}{\rm I}$;
  \item[--] Compute $T_{z{\rm I}}^{\circ n+1}$ from $T_{z{\rm I}}^{\circ n}$ as in equation (\ref{Eq:Recurrence});
  \item[--] Stop at $n=N$ and add $T_{z{\rm I}}^{\circ N}$ to $S$;
  \end{itemize}
\item[-] Approximate $\calG_{\bf H}(z{\rm I})$ by $S/M$.
\end{itemize}
Observe that the above procedure have complexity $O(MNd^2)$. To speed up the algorithm some tailor made modifications were done. See the next section for details.

At a first sight the previous routine may seem quite unsatisfactory as it depends on Monte-Carlo method. However, suppose that we have a general formula for $\calG_{\bf H}$ that depends on the joint distribution of the entries of $A$. If such a formula comes in the form of an integral with respect to this joint distribution then $\calG_{\bf H}(z)$ would be an integral over a $d(d+1)/2$-dimensional space. Even for $d$ relatively small, such an integral is likely to be evaluated by a Monte-Carlo-like method. Our point here is that in general the previous routine is as far as we can go numerically speaking.

We also have the following corollaries.

\begin{corollary}
\label{Corollary:InvariantDistribution}
Let ${\bf H}=A\circ{\bf X}$ be a semicircular mixture. If for each $\omega\in\Omega$ there exists $K(\omega)$ such that $K(\omega)=\sum_{j=1}^d A_{i,j}^2$ for all $i\in\{1,\ldots,d\}$, then $\calG_{\bf H}(z{\rm I})$ is given by
\begin{equation*}
\calG_{\bf H}(z{\rm I}) = \E{\dfrac{z-\sqrt{z^2-4K}}{2K}}{\rm I}.
\end{equation*}
In particular, if $K$ does not depend on $\omega$ then $\calG_{\bf H}(z{\rm I})$ is given by $\calG_{\bf H}(z{\rm I}) = \dfrac{z-\sqrt{z^2-4K}}{2K}{\rm I}$.
\end{corollary}

The entries of $A$ were supposed to be positive for simplicity, but nothing stop us from using the results found so far to any matrix as long as the squares are replaced by the appropriate square norms.

\begin{corollary}
Let ${\bf H}=A\circ{\bf X}$ be a semicircular mixture and $F$ its mean analytical distribution. If $A$ is a unitary selfadjoint random matrix then $F=F^{\bf X}$.
\end{corollary}

In the particular case when $A$ is constant the previous corollary is a remarkable, and already known, property of semicircular elements over $\Mat{d}{\C}$. In this sense, this corollary asserts that the random version of this known fact is also true. This makes reasonable that the structure of the moments of the CLT limit in $L^{\infty-}(\Omega,\calF,\mathbb{P})\otimes\Mat{d}{\calA}$ is connected to those in the operator-valued case.

\section{Central Limit Theorem}

Suppose that ${\bf Z}\in L^{\infty-}(\Omega,\calF,\mathbb{P})\otimes\Mat{d}{\calA}$ can be written as $\sum_{i=1}^n a_i\otimes {\bf Z}_i$ for some $n\in\N$, $a_1,\ldots,a_n\in L^{\infty-}(\Omega,\calF,\mathbb{P})$ and ${\bf Z}_1,\ldots,{\bf Z}_n\in\Mat{d}{\calA}$. We construct $\{{\bf Z}^{(j)}\}_j$ independent and free over $\Mat{d}{\C}$ copies of ${\bf Z}$ by a) creating $\{a^{(j)}_1,\ldots,a^{(j)}_n\}_j$ independent identically distributed copies of $\{a_1,\ldots,a_n\}$, b) creating $\{{\bf Z}^{(j)}_1,\ldots,{\bf Z}^{(j)}_n\}_j$ free over $\Mat{d}{\C}$ copies of $\{{\bf Z}_1,\ldots,{\bf Z}_n\}$ and c) defining ${\bf Z}^{(j)}=\sum_{i=1}^n a^{(j)}_1\otimes{\bf Z}^{(j)}_i$.

We now establish a central limit theorem for random operator-valued matrices. The proof of the following central limit theorem is analogous to the operator-valued case, though it shows that independence, freeness over $\Mat{d}{\C}$ and centeredness assumptions\footnote{In the sense of definitions \ref{Def:IndependenceFreeness} and \ref{Def:Centered}.} imply that the limiting mean $\Mat{d}{\C}$-moments have the same structure as the $\Mat{d}{\C}$-moments of semicircular element in $\Mat{d}{\calA}$. This is a quite remarkable feature as the non-trivial extension $L^{\infty-}(\Omega,\calF,\mathbb{P})\otimes\Mat{d}{\calA}$ of $\Mat{d}{\calA}$ have the same structure for the mean $\Mat{d}{\C}$-moments of the CLT limit as the CLT limit in $\Mat{d}{\calA}$ itself.

\begin{theorem}
\label{Thm:CentralLimitTheorem}
Let ${\bf Z}\in L^{\infty-}(\Omega,{\mathcal F},{\mathbb P}) \otimes \Mat{d}{\calA}$ be a selfadjoint centered random operator-valued matrix. Suppose that $\{{\bf Z}^{(n)}\}_{n\geq1}$ are independent and free over $\Mat{d}{\C}$ copies of ${\bf Z}$. Then the normalized sum ${\bf S}_N = \dfrac{{\bf Z}^{(1)}+\cdots+{\bf Z}^{(N)}}{\sqrt{N}}$ satisfies
\begin{equation*}
\lim_{N\to\infty} \E{\Eop{{\bf S}_N^m}} = \sum_{\pi\in\NC_2(m)} \hat{\kappa}_\pi({\bf Z})
\end{equation*}
for all $m\in\N$ where $\hat{\kappa}(\pi)$ is defined as in Definition \ref{Def:SemicircularElement} using $\hat{\eta}(B)=\E{\Eop{{\bf Z}B{\bf Z}}}$ ($B\in\Mat{d}{\C}$) instead of $\eta$.
\end{theorem}

Here is worth to point out the following. Consider $d=1$ and the non-commutative probability space $(L^{\infty-}(\Omega,\calF,\mathbb{P})\otimes \calA,\hat{\FI}:=\mathbb{E}\otimes\FI)$. Suppose $a_1,a_2$ are independent symmetric Bernoulli random variables in $L^{\infty-}(\Omega,\calF,\mathbb{P})$ and $b_1,b_2\in\calA$ are free (w.r.t. $\FI$) Poisson non-commutative random variables. By definition, $a_1\otimes b_1$ and $a_2\otimes b_2$ are independent and free over $\Mat{1}{\C}$. Observe that
\begin{align*}
\hat{\FI}((a_1\otimes b_1)(a_2\otimes b_2)(a_1\otimes b_1)(a_2\otimes b_2)) &= \E{a_1^2a_2^2}\F{b_1b_2b_1b_2}\\
&= \F{b_1^2}\F{b_2}^2+\F{b_1}^2\F{b_2^2}-\F{b_1}^2\F{b_2}^2 = 7,
\end{align*}
but $\hat{\FI}((a_1\otimes b_1)^2)\hat{\FI}(a_2\otimes b_2)^2+\hat{\FI}(a_1\otimes b_1)^2\hat{\FI}((a_2\otimes b_2)^2)-\hat{\FI}(a_1\otimes b_1)^2\hat{\FI}(a_2\otimes b_2)^2 = 0$ as $\E{a_1}=\E{a_2}=0$ and so $\hat{\FI}(a_1\otimes b_1)=\hat{\FI}(a_2\otimes b_2)=0$. In particular, $a_1\otimes b_1$ and $a_2\otimes b_2$ are not free w.r.t. $\hat{\FI}$. Therefore the previous theorem is a different result from the operator-valued central limit theorem, even when it gives us the same structure for the $\Mat{d}{\C}$-moments.

As a particular consequence of the central limit theorem derived we have the following.

\begin{corollary}
\label{Corollary:CLT}
Let $\{A^{(n)}\}_{n\geq1}$ be i.i.d. $d\times d$ selfadjoint random matrices with common distribution $A$ with entries in $L^{\infty-}(\Omega,\calF,\mathbb{P})$ and let $\{{\bf X}^{(n)}\}_{n\geq1}$ be selfadjoint operator-valued centered elements free over $\Mat{d}{\C}$ with common distribution ${\bf X}$. Then the normalized sum ${\bf S}_N = \dfrac{A^{(1)}\circ{\bf X}^{(1)}+\cdots+A^{(N)}\circ{\bf X}^{(N)}}{\sqrt{N}}$ satisfies
\begin{equation*}
\lim_{N\to\infty} \E{\Eop{{\bf S}_N^m}} = \E{\Eop{{\bf Y}^m}}
\end{equation*}
for all $m\in\N$ where ${\bf Y}\in\Mat{d}{\calA}$ is an operator-valued semicircular element with covariance given by
\begin{equation}
\label{Eq:Covariance}
\F{{\bf Y}_{i,k}{\bf Y}_{j,l}} = \E{A_{i,k}A_{j,l}}\F{{\bf X}_{i,k}{\bf X}_{j,l}}.
\end{equation}
\end{corollary}

Observe that in the previous corollary we do not assume any particular distribution for ${\bf X}$.

\section{Numerical Example}
\label{Section:NumericalExample}

To illustrate the technique discussed after Theorem \ref{Thm:Elim}, we compute the following example. We take the empirical eigenvalue distribution of $1000$ matrices of size $300\times 300$ with the following distribution: $d=3$ and $N=100$; the operator-valued part ${\bf X}$ was approximated by $(W+W^*)/20$ where $W$ is a $300\times 300$ complex Gaussian random matrix with i.i.d. entries with 0 mean and variance 2; the random part $A$ was taken as a $3\times 3$ matrix with i.i.d. standard Rayleigh distributed entries, up to symmetries. To approximate $\calG_{\bf H}(z{\rm I})$ we used the Monte-Carlo technique with 1000 iterations. To speed up the iterative method it did not stop at a fixed $N$ but when the difference between two consecutive iterations was smaller than $10^{-3}$ in $||\cdot||_1$. To make the iterative method even faster, given a realization of $A$, we computed $G_{\bf H}(z{\rm I})$ for all the desired values of $z$, say $z[1],\ldots,z[L]$, using as initial value to iterate $T_{z[n+1]{\rm I}}$ the last value $T_{z[n]{\rm I}}^{\circ N}$.

We took $\Im(z[n])=0.001$. In the author experience, this value is small enough to provide a good approximation, as showed in the following picture, and at the same time is far enough from the real axis to ensure convergence of the iterative method. Needless to say, the figure shows good agreement between the observed eigenvalues and the estimated density computed from $\calG_{\bf H}$.

\begin{figure}[ht]
 \centering
 \includegraphics[width=0.40\textwidth,height=0.25\textwidth,keepaspectratio=false]{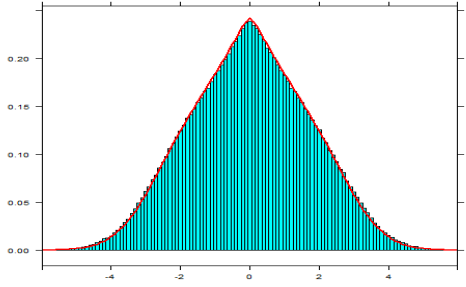}
 \caption{Histogram of the eigenvalues and plot of the estimated theoretical density.}
 \label{Fig:Hist}
\end{figure}

\section{Proofs of the Main Results}
\label{Section:Proofs}

\begin{proof}[{\bf Proof of Proposition \ref{Prop:Fubini}}]
It is well known that
\begin{equation*}
\int_\R f(t) \dif F(t) = \E{\int_\R f(t) \dif F^{\bf H}(t)}
\end{equation*}
for all positive or $F$-integrable function $f$. Let $z\in\C^+$ be fixed. By definition $g_F(z)=\displaystyle \int_{-\infty}^\infty \frac{1}{z-t} \dif F(t)$. Since $|z-t|^{-1}\leq\Im(z)^{-1}$, the function $(z-t)^{-1}$ is bounded and in particular integrable w.r.t. $F$. Thus
\begin{equation*}
g_F(z) = \E{\int_{-\infty}^\infty \dfrac{1}{z-t} \dif F^{\bf H}(t)}.
\end{equation*}
By definition of analytical distribution, the integral inside the expectation is equal to $\trd{G_{\bf H}(z{\rm I})}$ and therefore
\begin{align*}
g_F(z) &= \E{\trd{G_{\bf H}(z{\rm I})}}\\
&= \trd{\E{G_{\bf H}(z{\rm I})}}\\
&= \trd{\calG_{\bf H}(z{\rm I})},
\end{align*}
where the middle equality follows from the fact that trace and expected value commute.

For $\omega\in\Omega$ fixed the values of $\{A_{i,j} \mid 1\leq i,j\leq d\}$ are also fixed and finite a.s., since these are the variances of the entries of ${\bf H}(\omega)$ we have then that the support of $F^{{\bf H}(\omega)}$ is compact and also $\trd{\Eop{{\bf H}(\omega)^m}}=\int_\R t^m \dif F^{{\bf H}(\omega)}(t)$ for all $m\in\N$ \cite{Mingo2014}. Therefore, for $m\in\N$,
\begin{align*}
\E{\trd{\Eop{{\bf H}^m}}} &= \E{\int_\R t^m \dif F^{\bf H}(t)}\\
&= \int_\R t^m \dif F(t),
\end{align*}
where the last equation, as in the previous paragraph, is true as long as $t^m$ is integrable w.r.t. $F$. We will prove that $t^m$ is integrable w.r.t. $F$.

Let $m\in\N$ be fixed. Since $|t|^m$ is positive,
\begin{align*}
\int_\R |t|^m \dif F(t) = \E{\int_\R |t|^m \dif F^{\bf H}(t)} \leq \E{\lambda_{\max}({\bf H})^m}
\end{align*}
where $\lambda_{\max}({\bf H})(\omega)$ is the supremum of the absolute value over the support of $F^{{\bf H}(\omega)}$, i.e.
\begin{equation*}
\lambda_{\max}({\bf H})(\omega) = \sup_{\lambda\in\textnormal{Supp}(F^{{\bf H}(\omega)})} |\lambda|.
\end{equation*}
If we find a constant $C(\omega)$ such that $\int_\R t^n \dif F^{{\bf H}(\omega)}(t)\leq C(\omega)^n$ for all $n\in\N$, then we will have that $\lambda_{\max}({\bf H})(\omega)\leq C(\omega)$. Recall that $\int_\R t^n \dif F^{{\bf H}(\omega)}(t)=\trd{\Eop{{\bf H}(\omega)^n}}$ for $n\in\N$ and $\omega\in\Omega$ fixed. A straightforward computation shows that for $n\in\N$
\begin{align*}
\int_\R t^n \dif F^{{\bf H}(\omega)}(t) &= \trd{\Eop{{\bf H}(\omega)^n}} = \frac{1}{d} \sum_{i_1=1}^d \F{\left({\bf H}(\omega)^n\right)_{i_1,i_1}}\\
&= \frac{1}{d} \sum_{i_1,i_2,\ldots,i_n=1}^d \F{{\bf H}(\omega)_{i_1,i_2}\cdots{\bf H}(\omega)_{i_n,i_1}}\\
&= \frac{1}{d} \sum_{i_1,i_2,\ldots,i_n=1}^d (A_{i_1,i_2}\cdots A_{i_n,i_1})(\omega)\F{{\bf X}_{i_1,i_2}\cdots{\bf X}_{i_n,i_1}}.
\end{align*}
The Wick type formula for a free circular family shows that $\F{{\bf X}_{i_1,i_2}\cdots{\bf X}_{i_n,i_1}}\geq0$ for all $i_1,\ldots,i_n\in\{1,\ldots,d\}$. Let $M(\omega) = \max_{i,j=1,\ldots,d} A_{i,j}(\omega)$. By assumption $A_{i,j}$ is positive for all $1\leq i,j\leq d$, so the positivity of the coefficients $\F{{\bf X}_{i_1,i_2}\cdots{\bf X}_{i_n,i_1}}$ implies
\begin{align*}
\int_\R t^n \dif F^{{\bf H}(\omega)}(t) &\leq \frac{1}{d} \sum_{i_1,i_2,\ldots,i_n=1}^d M(\omega)^n \F{{\bf X}_{i_1,i_2}\cdots{\bf X}_{i_n,i_1}}\\
&= M(\omega)^n \trd{\Eop{{\bf X}^n}}.
\end{align*}
Since the analytical distribution of ${\bf X}$ is compact \cite{Mingo2014}, we have that $\lambda_{\max}({\bf X}) := \sup_{\lambda\in\textnormal{Supp}(F^{\bf X})} |\lambda|$ is finite and $\trd{\Eop{{\bf X}^n}}\leq \lambda_{\max}({\bf X})^n$. Thus
\begin{align*}
\int_\R t^n \dif F^{{\bf H}(\omega)}(t) &\leq M(\omega)^n \lambda_{\max}({\bf X})^n,
\end{align*}
and in particular $\lambda_{\max}({\bf H})\leq \lambda_{\max}({\bf X})M$. Therefore
\begin{align*}
\int_\R |t|^m \dif F(t) &\leq \E{\lambda_{\max}({\bf H})^m}\\
&\leq \lambda_{\max}({\bf X})^m\E{M^m} <\infty,
\end{align*}
where the existence of the $m$-th moment of the maximum is guaranteed by the fact that $A_{i,j}\in L^{\infty-}(\Omega,\calF,\mathbb{P})$ for all $1\leq i,j\leq d$.
\end{proof}

Recall the following theorem from \cite{Speicher2007}. We rephrase it in our terminology, so it constitutes the basis for the pointwise analysis (w.r.t. $\omega\in\Omega$).

\begin{theorem}
\label{Thm:Speicher}
Let ${\bf H}=A\circ{\bf X}$ be a semicircular mixture. Fix $\omega\in\Omega$. Define the mapping $\eta:\Mat{d}{\C}\to\Mat{d}{\C}$ given by $\eta(B)=\Eop{{\bf H}(\omega)B{\bf H}(\omega)}$. Then, for $B\in\Mat{d}{\C}^+$, the $\Mat{d}{\C}$-valued Cauchy transform of ${\bf H}(\omega)$ is given by
\begin{equation}
G_{{\bf H}(\omega)}(B) = \lim_{n\to\infty} T_B^{\circ n}(W)
\end{equation}
for any $W\in\Mat{d}{\C}^-$ where $T_B(W):=(B-\eta(W))^{-1}$. Moreover, $G_{{\bf H}(\omega)}(B)\in\Mat{d}{\C}^-$ and satisfies
\begin{equation}
\label{Eq:FixedPointEquation}
(B-\eta(G_{{\bf H}(\omega)}(B)))G_{{\bf H}(\omega)}(B) = {\rm I}.
\end{equation}
\end{theorem}

Before proving Theorem \ref{Thm:Elim} we need to prove the following lemmas. The next lemma shows that the definition of $\eta$ in Definition \ref{Def:Eta} actually coincides with the definition of $\eta$ in the previous theorem for semicircular mixtures.

\begin{lemma}
\label{Lemma:eta}
Let ${\bf H}=A\circ{\bf X}$ be a semicircular mixture and let $\eta:\Mat{d}{\C}\to\Mat{d}{\C}$ be the random mapping defined by $\eta(B)=\Eop{{\bf H}B{\bf H}}$ for $B\in\Mat{d}{\C}$. If $B\in\Mat{d}{\C}$ then $\eta(B)\in\Diag{d}{\C}$ and
\begin{equation*}
\eta(B)_{i,i} = \sum_{j=1}^d A_{i,j}^2 B_{j,j}.
\end{equation*}
\end{lemma}

\begin{proof}
Let $B\in\Mat{d}{\C}$ and $i,k\in\{1,\ldots,d\}$. By definition of $\EOP$ we have that $\eta(B)_{i,k} = \F{({\bf H}B{\bf H})_{i,k}}$ and a straightforward computation shows that
\begin{align*}
\eta(B)_{i,k} &=\sum_{i_1,i_2=1}^d \F{{\bf H}_{i,i_1} B_{i_1,i_2} {\bf H}_{i_2,k}}\\
&=\sum_{i_1,i_2=1}^d A_{i,i_1} A_{i_2,k} B_{i_1,i_2} \F{{\bf X}_{i,i_1}{\bf X}_{i_2,k}}.
\end{align*}
From the fact that $\{{\bf X}_{i,k} \mid 1\leq i,k\leq d\}$ is a free circular family with ${\bf X}_{i,j}=({\bf X}_{j,i})^*$, we conclude that
\begin{equation*}
\eta(B)_{i,k} = \delta_{i,k} \sum_{j=1}^d A_{i,j}^2 B_{j,j}.
\end{equation*}
Therefore $\eta(B)\in\Diag{d}{\C}$ and the claimed expression holds.
\end{proof}

The previous lemma, in notation of Theorem \ref{Thm:Speicher}, shows that $T_{z{\rm I}}(D)\in\Diag{d}{\C}$ for all $D\in\Diag{d}{\C}$. This easily implies the following corollary.

\begin{corollary}
\label{Corollary:ImaginaryPartG}
Let ${\bf H}$ be a semicircular mixture, then $G_{\bf H}(z{\rm I})\in\Diag{d}{\C}$ and thus $G_{\bf H}(z{\rm I})_{i,i}\in\C^-$ for all $1\leq i\leq d$.
\end{corollary}

It is important to notice that the following corollaries are weaker than the analysis done in \cite{Speicher2007}, however they are enough to prove Theorem \ref{Thm:Elim} so we include them for completeness.

\begin{lemma}
Let $z\in\C^+$, then $||G_{\bf H}(z{\rm I}))||_1\leq d\Im(z)^{-1}$.
\end{lemma}

\begin{proof}
Let $z\in\C^+$ be fixed and $i\in\{1,\ldots,d\}$. By Lemma \ref{Lemma:eta} and the fixed point equation (\ref{Eq:FixedPointEquation}) we have
\begin{align*}
1 &= \left(z-\eta(G_{\bf H}(z{\rm I}))_{i,i}\right) G_{\bf H}(z{\rm I})_{i,i}\\
&= \left(z-\sum_{j=1}^d A_{i,j}^2 G_{\bf H}(z{\rm I})_{j,j}\right)G_{\bf H}(z{\rm I})_{i,i}.
\end{align*}
By Corollary \ref{Corollary:ImaginaryPartG} we have that $\Im G_{\bf H}(z{\rm I})_{j,j} \leq 0$ for all $1\leq j\leq d$, therefore
\begin{align*}
|G_{\bf H}(z{\rm I})_{i,i}| &= \left|z-\sum_{j=1}^d A_{i,j}^2 G_{\bf H}(z{\rm I})_{j,j}\right|^{-1}\\
&\leq \Im\left(z-\sum_{j=1}^d A_{i,j}^2 G_{\bf H}(z{\rm I})_{j,j}\right)^{-1}\\
&\leq \Im(z)^{-1},
\end{align*}
from which the result follows.
\end{proof}

The proof of the following lemma follows the same lines as the previous one.

\begin{lemma}
Let $D\in\Mat{d}{\C}^-$, then $||T_{z{\rm I}}(D)||_1\leq d\Im(z)^{-1}$.
\end{lemma}

Now we are ready to proof Theorem \ref{Thm:Elim}.

\begin{proof}[{\bf Proof of Theorem \ref{Thm:Elim}}]
By Theorem \ref{Thm:Speicher}, for each $\omega\in\Omega$ we have that ${\bf H}$ satisfies $G_{\bf H}(z{\rm I}) = \lim_{n\to\infty} T_{z{\rm I}}^{\circ n}({\rm i}{\rm I})$. By the definition of $\calG$ we have that $\calG_{\bf H}(z{\rm I}) = \E{G_{\bf H}(z{\rm I})} = \E{\lim_{n\to\infty} T_{z{\rm I}}^{\circ n}({\rm i}{\rm I})}$, so by the previous lemmas and the dominated convergence theorem,
\begin{equation*}
\calG_{\bf H}(z{\rm I}) = \lim_{n\to\infty} \E{T_{z{\rm I}}^{\circ n}({\rm i}{\rm I})},
\end{equation*}
as required.
\end{proof}

\begin{proof}[{\bf Proof of Corollary \ref{Corollary:InvariantDistribution}}]
Let $\omega\in\Omega$ fixed. Observe that, for $z\in\C^+$, $\eta(z{\rm I}) = Kz{\rm I}$. Thus the fixed point equation (\ref{Eq:FixedPointEquation}) implies that $KG_{\bf H}(z{\rm I})^2 - zG_{\bf H}(z{\rm I})+{\rm I}=0$. Equivalently, we have that $G_{\bf H}(z{\rm I})=\dfrac{z-\sqrt{z^2-4K}}{2K}{\rm I}$ and therefore $\calG_{\bf H}(z{\rm I}) = \E{G_{\bf H}(z{\rm I})} = \E{\dfrac{z-\sqrt{z^2-4K}}{2K}}{\rm I}$ as claimed.
\end{proof}

Finally, we prove the central limit theorem, but first we have to prove the following lemma.

\begin{lemma}
\label{Lemma:PairCancellation}
Let ${\bf Z}\in L^{\infty-}(\Omega,\calF,\mathbb{P})\otimes\Mat{d}{\calA}$ be a selfadjoint centered random operator-valued matrix. If ${\bf Z}^{(1)}$ and ${\bf Z}^{(2)}$ are independent and free over $\Mat{d}{\C}$ copies of ${\bf Z}$, then
\begin{equation*}
\E{\Eop{{\bf Z}^{(1)}{\bf Z}^{(2)}B{\bf Z}^{(2)}{\bf Z}^{(1)}}} = \hat{\eta}(\hat{\eta}(B))
\end{equation*}
where
\begin{equation*}
\hat{\eta}(B) = \E{\Eop{{\bf Z}B{\bf Z}}}.
\end{equation*}
\end{lemma}

\begin{proof}
Since ${\bf Z}^{(1)}$ and ${\bf Z}^{(2)}$ are free over $\Mat{d}{\C}$,
\begin{align*}
\E{\Eop{{\bf Z}^{(1)}{\bf Z}^{(2)}B{\bf Z}^{(2)}{\bf Z}^{(1)}}} = \E{\Eop{{\bf Z}^{(1)}\Eop{{\bf Z}^{(2)}B{\bf Z}^{(2)}}{\bf Z}^{(1)}}}.
\end{align*}
By the tower property of the conditional expectation
\begin{align*}
\E{\Eop{{\bf Z}^{(1)}{\bf Z}^{(2)}B{\bf Z}^{(2)}{\bf Z}^{(1)}}} &= \E{\E{\Eop{{\bf Z}^{(1)}\Eop{{\bf Z}^{(2)}B{\bf Z}^{(2)}}{\bf Z}^{(1)}}\bigg| {\bf Z}^{(1)}}}\\
&= \E{\Eop{{\bf Z}^{(1)}\E{\Eop{{\bf Z}^{(2)}B{\bf Z}^{(2)}}\bigg| {\bf Z}^{(1)}}{\bf Z}^{(1)}}},
\end{align*}
where the last equation uses that conditional expectation and $\EOP$ commute (see Section \ref{Section:Introduction}). By independence,
\begin{align*}
\E{\Eop{{\bf Z}^{(1)}{\bf Z}^{(2)}B{\bf Z}^{(2)}{\bf Z}^{(1)}}} &= \E{\Eop{{\bf Z}^{(1)}\E{\Eop{{\bf Z}^{(2)}B{\bf Z}^{(2)}}}{\bf Z}^{(1)}}}\\
&= \hat{\eta}(\hat{\eta}(B)),
\end{align*}
as required.
\end{proof}

This random operator-valued version of the pair cancellation lemma allows us to prove the central limit theorem for random operator-valued matrices {\it mutatis mutandis} as in the operator-valued case.

\begin{proof}[{\bf Proof of Theorem \ref{Thm:CentralLimitTheorem}}]
Let $m\in\N$ be fixed, then
\begin{align*}
\E{\Eop{{\bf S}_N^m}} &= N^{-m/2} \sum_{i_1,\ldots,i_m=1}^N \E{\Eop{{\bf Z}^{(i_1)} \cdots {\bf Z}^{(i_m)}}}.
\end{align*}
As in the scalar and operator-valued cases, see \cite{Nica2006} and \cite{Speicher1998} respectively, the independence, freeness and identically distributed assumptions imply that the value of $\E{\Eop{{\bf Z}^{(i_1)} \cdots {\bf Z}^{(i_m)}}}$ depends on $i_1,\ldots,i_m$ by means of which indices are equal and which are different. Let $\mathcal{P}(m)$ denote the set of all partitions of the integers $\{1,\ldots,m\}$. For $\pi\in\mathcal{P}(n)$, we denote by $B_\pi^N$ the number of tuples $(i_1,\ldots,i_m)$ of indices such that $1\leq i_1,\ldots,i_m\leq N$ and $i_j=i_k$ if and only if $j$ and $k$ belong to the same block in $\pi$. Also, we denote by $C_\pi$ the value of $\E{\Eop{{\bf Z}^{(i_1)} \cdots {\bf Z}^{(i_m)}}}$ where $(i_1,\ldots,i_m)$ satisfies that $i_j=i_k$ if and only if $j$ and $k$ belong to the same block in $\pi$. We have then
\begin{equation*}
\E{\Eop{{\bf S}_N^m}} = N^{-m/2} \sum_{\pi\in\mathcal{P}(m)} B_\pi^N C_\pi.
\end{equation*}
As in the aforementioned references, we can analyze $\mathcal{P}(m)$ in four groups: partitions with singletons, non-crossing pairings, crossing pairings and {\it the rest}. Applying $\EOP$, the freeness and centeredness assumptions imply that partitions with singletons and crossing pairings vanish. A simple combinatorial analysis shows that {\it the rest} does not contribute asymptotically. Therefore, just the non-crossing partitions contribute asymptotically, with $N^{-m/2}B_\pi^N\to1$ as $N\to\infty$, and the previous pair cancellation lemma then leads to
\begin{align*}
\E{\Eop{{\bf S}_N^m}} &= \sum_{\pi\in\NC_2(m)}  \hat{\kappa}_\pi({\bf Z})+o(1).
\end{align*}
This establishes the desired convergence.
\end{proof}

It is important to notice that the independence were used only to apply the pair cancellation lemma (Lemma \ref{Lemma:PairCancellation}). Since the properties of $\EOP$ were used most of the time, it is natural then to expect that the limiting $\Mat{d}{\C}$-moments of the CLT limit have the same structure as in the operator-valued case.

\begin{proof}[{\bf Proof of Corollary \ref{Corollary:CLT}}]
The hypothesis of the CLT are clearly satisfied and an easy computation shows that $\displaystyle \hat{\eta}_{A\circ{\bf X}}(B)_{i,j} = \sum_{i_1,i_2=1}^d B_{i_1,i_2} \E{A_{i,i_1}A_{i_2,j}} \F{X_{i,i_1}X_{i_2,j}}$, which establishes equation (\ref{Eq:Covariance}).
\end{proof}

\end{document}